\documentclass[a4paper,12pt,english]{smfart}

\usepackage{amssymb}
\usepackage{url}
\RequirePackage{mathrsfs}
\DeclareSymbolFont{rsfscript}{OMS}{rsfs}{m}{b}
\DeclareSymbolFontAlphabet{\mathrsfs}{rsfscript}
\usepackage[utopia,expert]{mathdesign}
\usepackage{lscape}
\usepackage{calligra}
\usepackage[T1]{fontenc}

\usepackage{frcursive}

\usepackage{palatino}
\usepackage{rotating}
\usepackage{graphicx}

\usepackage{enumerate}

\usepackage{color}
\definecolor{shadecolor}{gray}{0.90}

\def\bfit{\bfseries\itshape}

\def\proofname{D\'emonstration}

\input xypic
\xyoption{all}
\xyoption{arc}

\newtheorem{theo}{Theorem}[section]
\def\thetheo{\thesection.\arabic{theo}}
\newtheorem{prop}[theo]{Proposition}

\newtheorem{coro}[theo]{Corollary}

\renewcommand\theequation{\thetheo}
\def\equat{\refstepcounter{theo}\begin{equation}}
\def\endequat{\end{equation}}

\renewcommand\thetable{\thetheo}

\renewcommand\thesection{\arabic{section}}
\renewcommand\thesubsection{\thesection.\Alph{subsection}}

\def\contentsname{Table des mati\`eres}
\def\listfigurename{Liste des figures}
\def\listtablename{Liste des tables}
\def\appendixname{Appendice}

\def\refname{R\'ef\'erences}

    \def\CM{{\mathbb{C}}}

  \def\pG{{\mathfrak p}}  
    \def\QM{{\mathbb{Q}}}

    \def\ZM{{\mathbb{Z}}}

    \def\BC{{\mathcal{B}}}

\def\Kb{{\mathbf K}}    
    
    \def\MC{{\mathcal{M}}}
    
    \def\OC{{\mathcal{O}}}

\def\Zrm{{\mathrm{Z}}}

  \def\eti{{\tilde{e}}}

          \def\eba{{\bar{e}}}

\def\b{\beta}
\def\g{\gamma}

\def\L{\Lambda}

\def\o{\omega}
\def\O{\Omega}

\def\s{\sigma}
\def\Sig{\Sigma}

               \def\Sigh{{\hat{\Sig}}}

\DeclareMathOperator{\Irr}{{\mathrm{Irr}}}

\DeclareMathOperator{\blocs}{{\mathrm{BlId}}}

\def\to{\rightarrow}
\def\longto{\longrightarrow}

\def\longtrait#1{\hspace{0.3em}{~ \SS{#1} ~ \over ~}\hspace{0.3em}}
\def\longtrait#1{\hspace{0.3em}\frac{~\SS{#1}}{~}\hspace{0.3em}}

\def\DS{\displaystyle}
\def\SS{\scriptstyle}

\def\lexp#1#2{\kern\scriptspace\vphantom{#2}^{#1}\kern-\scriptspace#2}

\mathchardef\inferieur="321E
\mathchardef\superieur="321F

\def\eqna{\begin{eqnarray*}}
\def\endeqna{\end{eqnarray*}}

\catcode`\@=11
\long\def\@car#1#2\@nil{#1}
\long\def\@first#1#2{#1}
\long\def\@second#1#2{#2}
\long\def\ifempty#1{\expandafter\ifx\@car#1@\@nil @\@empty
  \expandafter\@first\else\expandafter\@second\fi}
\catcode`\@=12

\renewcommand{\tocsection}[3]{%
  {\bf \indentlabel{\ifempty{#2}{}{\ignorespaces#1 #2.\quad}}#3}}

\def\boitegrise#1#2{\begin{centerline}{\fcolorbox{black}{shadecolor}{~
    \begin{minipage}[t]{#2}{\vphantom{~}#1\vphantom{$A_{\DS{A_A}}$}}
            \end{minipage}~}}\end{centerline}\medskip}

\theoremstyle{remark}

\theoremstyle{plain}

\def\blocs{{\mathrm{Blocks}}}

\def\xyinj{\ar@{^{(}->}}
\def\xysur{\ar@{->>}}

\bigskip

\makeatletter
\def\hlinewd#1{%
\noalign{\ifnum0=`}\fi\hrule \@height #1 %
\futurelet\reserved@a\@xhline}
\makeatother

\newlength\epaisLigne

\usepackage{multirow}

\def\boitegrise#1#2{\begin{centerline}{\fcolorbox{black}{shadecolor}{~
    \begin{minipage}[t]{#2}{\vphantom{~}#1\vphantom{$A_{\DS{A_A}}$}}
            \end{minipage}~}}\end{centerline}\medskip}

\def\longiso{\hskip0.6mm{\xrightarrow{\sim}\hskip0.6mm}}
\def\longbij{\hskip0.6mm{{\stackrel{\sim}{\longleftrightarrow}}\hskip0.6mm}}

\begin{document}

\baselineskip=16pt
%\large\baselineskip=20pt
%\Large\baselineskip=24pt

\title{Blocks of the Grothendieck ring \\
of equivariant bundles on a finite group}

\author{{\sc C\'edric Bonnaf\'e}}
\address{
Institut Montpelli\'erain Alexander Grothendieck (CNRS: UMR 5149), 
Universit\'e Montpellier 2,
Case Courrier 051,
Place Eug\`ene Bataillon,
34095 MONTPELLIER Cedex,
FRANCE} 

\makeatletter
\email{cedric.bonnafe@univ-montp2.fr}
\makeatother

% \author{{\sc Rapha\"el Rouquier}}
% 
% \address{Mathematical Institute, 
% University of Oxford, 24-29 St Giles', Oxford, OX1 3LB, UK}
% \email{rouquier@maths.ox.ac.uk}

%\makeatother

%\subjclass{According to the 2000 classification:
%Primary ???; Secondary ???}

\date{\today}

\thanks{The author is partly supported by the ANR (Project No ANR-12-JS01-0003-01 ACORT)}

\begin{abstract} 
If $G$ is a finite group, the Grothendieck group $\Kb_G(G)$ of the category 
of $G$-equivariant $\CM$-vector bundles on $G$ (for the action of $G$ on itself by conjugation) is endowed 
with a structure of (commutative) ring. If $K$ is a sufficiently large extension of $\QM_{\! p}$ 
and $\OC$ denotes the integral closure of $\ZM_{\! p}$ in $K$, the $K$-algebra 
$K\Kb_G(G)=K \otimes_\ZM \Kb_G(G)$ is split semisimple. The aim of this paper is to describe 
the $\OC$-blocks of the $\OC$-algebra $\OC \Kb_G(G)$.
\end{abstract}

\maketitle

\pagestyle{myheadings}

\markboth{\sc C. Bonnaf\'e}{\sc Blocks of the Grothendieck ring of equivariant vector bundles}

% \tableofcontents
% 
% \vskip1cm

\bigskip

\section{Notation, introduction}

\medskip

\subsection{Groups} 
We fix in this paper a finite group $G$, a prime number $p$ and a finite extension $K$ of 
the $p$-adic field $\QM_{\! p}$ such that $KH$ is split for all subgroups $H$ of $G$. We denote by $\OC$ 
the integral closure of $\ZM_{\! p}$ in $K$, by $\pG$ the maximal ideal of $\OC$, by 
$k$ the residue field of $\OC$ (i.e. $k=\OC/\pG$)
We denote by $\Irr(KG)$ the set of irreducible characters of $G$ (over $K$). 

A {\it $p$-element} (respectively {\it $p'$-element}) of $G$ is an element whose order is a power of $p$ 
(respectively prime to $p$). If $g \in G$, we denote by $g_p$ and $g_{p'}$ the unique 
elements of $G$ such that $g=g_pg_{p'}=g_{p'}g_p$, $g_p$ is a $p$-element and $g_{p'}$ 
is a $p'$-element. The set of $p$-elements (respectively $p'$-elements) of $G$ 
is denoted by $G_p$ (respectively $G_{p'}$).

If $X$ is a $G$-set (i.e. a set endowed with a left $G$-action), we denote by 
$[G\backslash X]$ a set of representatives of $G$-orbits in $X$. The reader can check 
that we will use formulas like 
$$\sum_{x \in [G \backslash X]} f(x)$$
(or families like $(f(x))_{x \in [G \backslash X]}$) only whenever $f(x)$ 
does not depend on the choice of the representative $x$ in its $G$-orbit. 
If $X$ is a set-$G$ (i.e. a set endowed with a right $G$-action), we define 
similarly $[X/G]$ and will use it according to the same principles.

\bigskip

\subsection{Blocks} 
A {\it block idempotent} of $k G$ (respectively $\OC G$) is a primitive idempotent of the 
center $\Zrm(kG)$ (respectively $\Zrm(\OC G)$) of $\OC G$. We denote by 
$\blocs(kG)$ (respectively $\blocs(\OC G$) the set of block idempotents of $kG$ 
(respectively $\OC G$). Reduction modulo $\pG$ induces a bijection 
$\blocs(\OC G) \longiso \blocs(kG)$, $e \mapsto \eba$ 
(whose inverse is denoted by $e \mapsto \eti$).

A {\it $p$-block} of $G$ is a subset $\BC$ of $\Irr(G)$ such that $\BC=\Irr(KG e)$, 
for some block idempotent $e$ of $\OC G$. 

\bigskip

\def\irrpaires{{\mathrm{IrrPairs}}}
\def\blocpaires{{\mathrm{BlPairs}}}
\def\fourier{{\mathrm{Four}}}
\def\fou#1{{ {\boldsymbol{\bigl\{}} \hskip0.5mm #1 \hskip0.5mm {\boldsymbol{\bigr\}}} }}
\def\bundles{{\BC{\!u\!n}}}

\subsection{Fourier coefficients}
Let 
$$\irrpaires(G)=\{(g,\g)~|~g \in G~\text{and}~\g \in \Irr(KC_G(g))\}$$
$$\blocpaires_p(G)=\{(s,e)~|~s \in G_{p'}~\text{and}~e \in \blocs(\OC C_G(s))\}.\leqno{\text{and}}$$
The group $G$ acts (on the left) on these two sets by conjugation. We set
$$\MC(G)=[G \backslash \irrpaires(G)]\qquad\text{and}\qquad
\MC^p(G)=[G \backslash \blocpaires(G)].$$

If $(g,\g)$, $(h,\eta) \in \irrpaires(G)$, 
we define, following Lusztig~\cite[2.5(a)]{lusztig},
$$\fou{(g,\g),(h,\eta)}=
\frac{1}{|C_G(g)|\cdot |C_G(h)|} \sum_{\substack{x \in G \\ xhx^{-1} \in C_G(g)}} 
\g(xhx^{-1}) \eta(x^{-1} g^{-1} x).$$
Note that $\fou{(g,\g),(h,\eta)}$ depends only on the $G$-orbit of $(g,\g)$ and on the $G$-orbit of $(h,\eta)$. 
%We define the {\it Fourier matrix} of $G$, and we denote by $\fourier(G)$, the matrix
%$$\fourier(G)=\Bigl(\fou{(g,\g),(h,\eta)}\Bigr)_{(g,\g),(h,\eta) \in \MC(G)}.$$

\bigskip

\subsection{Vector bundles} 
Except from Proposition~\ref{prop:blocks} below, all the definitions, all the results in this subsection 
can be found in~\cite[\S{2}]{lusztig}. 
We denote by $\bundles_G(G)$ the category of $G$-equivariant finite dimensional 
$K$-vector bundles on $G$ (for the action of $G$ by conjugation). Its Grothendieck 
group $\Kb_G(G)$ is endowed with a ring structure. 
For each $(g,\g) \in \MC(G)$, let $V_{g,\g}$ be the isomorphism class 
(in $\Kb_G(G)$) of the simple object in $\bundles_G(G)$ associated with $(g,\g)$, 
as in~\cite[\S{2.5}]{lusztig} (it is denoted $U_{g,\g}$ there). 
Then
$$\Kb_G(G)=\bigoplus_{(g,\g) \in \MC(G)} \ZM V_{g,\g}.$$
The $K$-algebra $K \Kb_G(G)=K \otimes_\ZM \Kb_G(G)$ is split semisimple and commutative. 
Its simple modules (which have dimension one) are also parametrized by $\MC(G)$: 
if $(g,\g) \in \MC(G)$, the $K$-linear map  
$$\Psi_{g,\g} : K\Kb_G(G) \longto K$$
defined by
$$\Psi_{g,\g}(V_{h,\eta}) = \frac{|C_G(g)|}{\g(1)} 
\cdot \fou{(h^{-1},\eta),(g,\g)}$$
is a morphism of $K$-algebras and all morphisms of $K$-algebras 
$K\Kb_G(G) \longto K$ are obtained in this way.

We define similarly block idempotents of $k\Kb_G(G)$ and $\OC \Kb_G(G)$, 
as well as {\it $p$-blocks of} $\MC(G) \longbij \Irr(K\Kb_G(G))$. 
\medskip

\def\brauer{{\mathrm{Br}}}

\subsection{Brauer maps}
Let $\L$ denote one of the two rings $\OC$ or $k$. 
If $g \in G$ (and if we set $s=g_{p'}$), we denote by $\brauer_g^\L$ the $\L$-linear map 
$$\brauer_g^\L : \L C_G(s) \longto \L C_G(g)$$
such that 
$$\brauer_g^\L(h) =
\begin{cases}
h & \text{if $h \in C_G(g)$},\\
0 & \text{if $h \not\in C_G(g)$},
\end{cases}$$
for all $h \in C_G(s)$. Recall~\cite[Lemma~15.32]{isaacs} that 
\equat\label{eq:brauer}
\text{\it $\brauer_g^k$ induces a {\bfit morphism of algebras} $\Zrm(kC_G(s)) \to \Zrm(kC_G(g))$.}
\endequat
Therefore, if $e \in \blocs(\OC C_G(s))$, then $\brauer_g^k(e)$ is 
an idempotent of $\Zrm(k C_G(g))$ (possibly equal to zero) and we can write it a sum 
$\brauer_g^k(e)=e_1 + \cdots + e_n$, where $e_1$,\dots, $e_n$ are pairwise distinct 
block idempotents of $kC_G(g)$. We then set
$$\b_g^\OC(e)=\sum_{i=1}^n \eti_i.$$
It is an idempotent (possibly equal to zero, possibly non-primitive) of $\Zrm(\OC C_G(g))$.

\medskip

\subsection{The main result}\label{sub:main} 
In order to state more easily our main result, it will be more convenient (though it is not 
strictly necessary) to fix a particular set of representatives of conjugacy classes of $G$.

\medskip

\boitegrise{{\noindent \bf Hypothesis and notation.} {\it 
From now on, and until the end of this paper, we denote by:
\begin{itemize}
\item[$\bullet$] $[G_{p'}/\!\sim]$ a set of representatives of conjugacy classes of $p'$-elements in $G$.
\item[$\bullet$] $[G/\!\sim]$ a set of representatives of conjugacy classes of elements of $G$ such that, 
for all $g \in [G/\!\sim]$, $g_{p'} \in [G_{p'}/\!\sim]$.
\end{itemize}
We also assume that, if $(g,\g) \in \MC(G)$ or $(s,e) \in \MC^p(G)$, then 
$g \in [G/\!\sim]$ and $s \in [G_{p'}/\!\sim]$. %\\
$$\longtrait{\hphantom{AAAAAAAAAAAA}}$$
\hphantom{Aa} If $(s,e) \in \MC^p(G)$, we define $\BC_G(s,e)$ to be the set of pairs 
$(g,\g) \in \MC(G)$ such that:
\begin{itemize}
\item[(1)] $g_{p'}=s$.
\item[(2)] $\g \in \Irr(K C_G(g)\b_g^\OC(e))$.
\end{itemize}}
\vskip-0.4cm}{0.68\textwidth}

\bigskip

\begin{theo}\label{theo:main}
The map $(s,e) \mapsto \BC_G(s,e)$ induces a bijection between $\MC^p(G)$ to the set of $p$-blocks 
of $\MC(G)$.
\end{theo}

\section{Proof of Theorem~\ref{theo:main}}

\medskip

\subsection{Central characters and congruences}
If $(g,\g) \in \irrpaires(G)$, we denote by $\o_{g,\g} : \Zrm(KC_G(g)) \to K$ the {\it central character} 
associated with $\g$ (if $z \in \Zrm(KC_G(g))$, then $\o_{g,\g}(z)$ is the scalar through which $z$ acts 
on an irreducible $KC_G(g)$-module affording the character $\g$). It is a morphism of algebras: 
when restricted to $\Zrm(\OC C_G(g))$, it has values in $\OC$. 

If $h \in C_G(g)$, we denote by $\Sig_g(h)$ conjugacy class of $h$ in $C_G(g)$ and we set
$$\Sigh_g(h)=\sum_{v \in \Sig_g(h)} v \in \Zrm(\OC C_G(g)).$$
We have
\equat\label{eq:omega}
\o_{g,\g}\bigl(\Sigh_g(h)\bigr) = \frac{|\Sig_g(h)| \cdot \g(h)}{\g(1)}.
\endequat
We also recall the following classical results:

\bigskip

\begin{prop}\label{prop:blocs-kg}
If $g \in G$ and $\g$, $\g'$ are two irreducible characters of $C_G(g)$, then 
$\g$ and $\g'$ lie in the same $p$-block of $C_G(g)$ if and only if 
$$\o_{g,\g}(\Sigh_g(h)) \equiv \o_{g,\g'}(\Sigh_g(h)) \mod \pG$$
for all $h \in C_G(g)$.
\end{prop}

\begin{prop}\label{prop:blocks}
Let $(g,\g)$ and $(g',\g')$ be two elements of $\MC(G)$. Then $(g,\g)$ and $(g',\g')$ 
belong to the same $p$-block of $\MC(G)$ if and only if 
$$\Psi_{g,\g}(V_{h,\eta}) \equiv \Psi_{g',\g'}(V_{h,\eta}) \mod \pG$$
for all $(h,\eta) \in \MC(G)$.
\end{prop}

\bigskip

\subsection{Around the Brauer map}
As Brauer maps are morphisms of algebras, we have 
$$\sum_{e \in \blocs(kC_G(g_{p'}))} \brauer_g^p(e) = 1,$$
and so
\equat\label{eq:partition}
\text{\it The family $\bigl(\BC_G(g,e)\bigr)_{(g,e) \in \MC^p(G)}$ is a partition of $\MC(G)$.}
\endequat

Now, let $(g,\g) \in \MC(G)$ and let $s=g_{p'}$. 
If $e \in \blocs(\OC C_G(s))$ is such that $\g \in \Irr(K C_G(g)\b_g^\OC(e))$, 
and if $\s \in \Irr(K C_G(s) e)$, then~\cite[Lemma~15.44]{isaacs} 
\equat\label{eq:brauer-central}
\o_{s,\s}(z) \equiv \o_{g,\g}\bigl(\brauer_g^\OC(z)\bigr) \mod \pG
\endequat
for all $z \in \Zrm(\OC C_G(s))$. 

\bigskip

\subsection{Rearranging the formula for $\Psi_{g,\g}$} 
If $(g,\g)$, $(h,\eta) \in \irrpaires(g)$ then 
\equat\label{eq:fourier mieux}
\Psi_{g,\g}(V_{h,\eta}) = \sum_{\substack{x \in [C_G(g)\backslash G /C_G(h)] \\ xhx^{-1} \in C_G(g)}}
\eta(x^{-1} g x) \o_{g,\g}\bigl(\Sigh_g(xhx^{-1})\bigr).
\endequat
\begin{proof}
By definition,
$$\Psi_{g,\g}(V_{h,\eta}) = \frac{1}{\g(1)\cdot |C_G(h)|} 
\sum_{\substack{x \in G \\ xhx^{-1} \in C_G(g)}} \!\!\eta(x^{-1} g x)\g(xhx^{-1}) 
= \frac{1}{\g(1)} \sum_{\substack{x \in [G/C_G(h)] \\ xhx^{-1} \in C_G(g)}} \eta(x^{-1} g x)\g(xhx^{-1}).$$ 
Now, if $x \in G$ is such that $xhx^{-1} \in C_G(g)$ and if $u \in C_G(g)$, then 
$$\eta\bigl((ux)^{-1} g (ux)\bigr)\g\bigl((ux)h(ux)x^{-1}\bigr)=\eta(x^{-1}gx)\g(xhx^{-1}).$$
So we can gather the terms in the last sum according to their $C_G(g)$-orbit. We get
$$\Psi_{g,\g}(V_{h,\eta}) = 
\sum_{\substack{x \in [C_G(g)\backslash G/C_G(h)] \\ xhx^{-1} \in C_G(g)}} 
\eta(x^{-1} g x)\frac{|C_G(g)|}{|C_G(g) \cap xC_G(h)x^{-1}|} \cdot \frac{\g(xhx^{-1})}{\g(1)}.
$$
But, for $x$ in $G$ such that $xhx^{-1} \in C_G(g)$, 
$$\frac{|C_G(g)|}{|C_G(g) \cap xC_G(h)x^{-1}|}=|\Sig_g(xhx^{-1})|,$$
so the result follows from~\ref{eq:omega}.
\end{proof}

\bigskip

\begin{coro}\label{coro:block-kgg-inclus}
Let $g \in [G/\sim]$ and let $\g$, $\g' \in \Irr(KC_G(g))$ lying in the same $p$-block of $C_G(g)$. 
Then $(g,\g)$ and $(g,\g')$ lie in the same $p$-block of $\MC(G)$.
\end{coro}

\begin{proof}
This follows from~\ref{eq:fourier mieux} and Proposition~\ref{prop:blocks}.
\end{proof}

\medskip

\bigskip

\subsection{$p'$-part}
Fix $(g,\g) \in \MC(G)$. Then it follows from~\ref{eq:fourier mieux} 
that, for all $\chi \in \Irr(KG)$, 
\equat\label{eq:e1}
\Psi_{g,\g}(V_{1,\chi}) = \chi(g).
\endequat

\bigskip

\begin{prop}\label{prop:p-partie}
Let $(g,\g)$ and $(h,\eta)$ be two elements in $\MC(G)$ which lie in the same $p$-block. Then 
$g_{p'}=h_{p'}$.
\end{prop}

\begin{proof}
By Proposition~\ref{prop:blocks} and Equality~\ref{eq:e1}, it follows from the hypothesis 
that
$$\chi(g) \equiv \chi(h) \mod \pG$$
for all $\chi \in \Irr(KG)$. Hence $g_{p'}$ and $h_{p'}$ are conjugate 
in $G$ (see~\cite[Proposition~2.14]{bonnafe kg}), so they are equal according to our conventions explained 
in~\S\ref{sub:main}.
\end{proof}

\bigskip

\begin{prop}\label{prop:p-partie-gamma}
Let $s \in G_{p'}$ and let $\s$, $\s' \in \Irr(KC_G(s))$. Then $(s,\s)$ and $(s,\s')$ 
lie in the same $p$-block if and only if $\s$ and $\s'$ lie in the same $p$-block of 
$C_G(s)$.
\end{prop}

\begin{proof}
The if part has been proved in Corollary~\ref{coro:block-kgg-inclus}. 
Conversely, assume that $(s,\s)$ and $(s,\s')$ lie in the same $p$-block. 
Fix $h \in C_G(s)$. Then $s \in C_G(h)$. Let $\eta_{s,h} : C_G(h) \to K$ 
be the class function on $C_G(h)$ defined by:
$$\eta_{s,h}(g)=
\begin{cases}
1 & \text{if $g_{p'}$ and $s$ are conjugate in $C_G(h)$,}\\
0 & \text{otherwise.}
\end{cases}$$
It follows from~\cite[Proposition~2.20]{bonnafe kg} that $\eta_{s,h} \in \OC \Irr(KC_G(h))$. Therefore, 
by~\ref{eq:fourier mieux} and Proposition~\ref{prop:blocks}, 
$$\sum_{\substack{x \in [C_G(s)\backslash G /C_G(h)] \\ xhx^{-1} \in C_G(s)}}
\eta_{s,h}(x^{-1} s x) \Bigl(\o_{s,\s}
\bigl(\Sigh_g(xhx^{-1}))-\o_{s,\s'}(\Sigh_g(xhx^{-1})\bigr)\Bigr) \equiv 0 \mod \pG.\leqno{(\#)}$$
Now, let $x \in G$ be such that $xhx^{-1} \in C_G(s)$. 
Since $x^{-1}sx$ is also a $p'$-element, $\eta_{s,h}(x^{-1}sx) = 1$ if and only if $s$ and $x^{-1}sx$ are 
conjugate in $C_G(h)$ that is, if and only if $x \in C_G(s)C_G(h)$. So it follows 
from $(\#)$ that
$$\o_{s,\s}(\Sigh_g(h)) \equiv \o_{s,\s'}(\Sigh_g(h))\mod \pG$$
for all $h \in C_G(s)$. This shows that $\s$ and $\s'$ lie in the same $p$-block of 
$C_G(s)$.
\end{proof}

\bigskip

\subsection{Last step} 
We shall prove here the last intermediate result:

\bigskip

\begin{prop}\label{prop:bse}
Let $(s,e) \in \MC^p(G)$ and let $(g,\g)$, $(g',\g') \in \BC_G(s,e)$. Then 
$(g,\g)$ and $(g',\g')$ are in the same $p$-block of $\MC(G)$.
\end{prop}

\medskip

\begin{proof}
We fix $\s \in \Irr(KC_G(s)e)$. It is sufficient to show that $(g,\g)$ and $(s,\s)$ 
are in the same $p$-block of $\MC(G)$. For this, let $(h,\eta) \in \MC(G)$. By 
Proposition~\ref{prop:p-partie}, we have $g_{p'}=s$, so $C_G(g) \subset C_G(s)$. 
So~\ref{eq:fourier mieux} can be rewritten:
$$\Psi_{g,\g}(V_{h,\eta})=
\sum_{x \in [C_G(s)\backslash G/C_G(h)]} 
\sum_{\substack{y \in [C_G(g)\backslash C_G(s) x C_G(h) /C_G(h)] \\ yhy^{-1} \in C_G(g)}} 
\eta(y^{-1}gy) \o_{g,\g}(\Sigh_g(yhy^{-1})).$$
Now, let $x \in [C_G(s)\backslash G/C_G(h)]$ and $y \in [C_G(g)\backslash C_G(s) x C_G(h) /C_G(h)]$ 
be such that $yhy^{-1} \in C_G(g)$. Then $yhy^{-1} \in C_G(s)$ and so $xhx^{-1} \in C_G(s)$. 
Moreover $y^{-1}sy$ is conjugate to $x^{-1}sx$ in $C_G(h)$. Finally, 
it is well-known (nd easy to prove) that $\eta(y^{-1}hy) \equiv \eta(y^{-1}sy) \mod \pG$ 
(see for instance~\cite[Proposition~2.14]{bonnafe kg}). Therefore:
$$\Psi_{g,\g}(V_{h,\eta}) \equiv 
\sum_{\substack{x \in [C_G(s)\backslash G/C_G(h)] \\ xhx^{-1} \in C_G(s)}} 
\eta(x^{-1} s x) ~\o_{g,\g}\Bigl(
\sum_{\substack{y \in [C_G(g)\backslash C_G(s) x C_G(h) /C_G(h)] \\ yhy^{-1} \in C_G(g)}} 
\Sigh_g(yhy^{-1}) \Bigr) \mod \pG.\leqno{(\diamondsuit)}$$
Now, let $x \in [C_G(s)\backslash G/C_G(h)]$ be such that $xhx^{-1} \in C_G(s)$. 
Then, by definition of the Brauer map, 
$$\brauer_g^\OC(\Sigh_s(xhx^{-1})) 
= \sum_{\substack{z \in [C_G(g)\backslash C_G(s)/(C_G(s) \cap C_G(xhx^{-1}))] \\
z(xhx^{-1})z^{-1} \in C_G(g)}} \Sigh_g((zx)h(zx)^{-1}).\leqno{(\heartsuit)}$$
But $(zx)_{z \in [C_G(g)\backslash C_G(s)/(C_G(s) \cap C_G(xhx^{-1}))]}$ is a set 
of representatives of double classes in $C_G(g)\backslash C_G(s) x C_G(h) / C_G(h)$. 
So it follows from $(\diamondsuit)$ and $(\heartsuit)$ that 
$$\Psi_{g,h} (V_{h,\eta}) \equiv 
\sum_{\substack{x \in [C_G(s)\backslash G/C_G(h)] \\ xhx^{-1} \in C_G(s)}} 
\eta(x^{-1} s x) ~\o_{g,\g}\bigl(\brauer_g^\OC(\Sigh_s(xhx^{-1}))\bigr).$$
Using now~\ref{eq:brauer-central} and~\ref{eq:fourier mieux}, we obtain that 
$$\Psi_{g,h} (V_{h,\eta}) \equiv \Psi_{s,\s}(V_{h,\eta}) \mod \pG,$$
as desired.
\end{proof}

\bigskip

\begin{proof}[Proof of Theorem~\ref{theo:main}]
Let $(s,e)$ and $(s',e')$ be two elements of $\MC^p(G)$ such that $\BC_G(s,e)$ and $\BC_G(s',e')$ 
are contained in the same $p$-block of $\MC(G)$ (see Proposition~\ref{prop:bse}). 
Let $\s \in \Irr\bigl(K C_G(s)e\bigr)$ and $\s' \in \Irr\bigl(K C_G(s')e'\bigr)$. 

Then $(s,\s)$ and $(s',\s')$ are in the same $p$-block, so it follows from 
Proposition~\ref{prop:p-partie} that $s=s'$ and it follows from Proposition~\ref{prop:p-partie-gamma} 
that $\g$ and $\g'$ are in the same $p$-block of $C_G(s)$, that is $e=e'$. This completes 
the proof of Theorem~\ref{theo:main}.
\end{proof}

\end{document}